\documentclass[12pt]{article}
\usepackage{amsthm}
\usepackage{tikz}
\usepackage{fullpage}
\usepackage{epstopdf}
\usepackage{subcaption}

\newtheorem{thm}{Theorem}
\newtheorem{lem}[thm]{Lemma}
\newtheorem{con}[thm]{Conjecture}

\begin{document}
\title{Conjectured bounds for the sum of squares of positive eigenvalues of a graph}
\author{Clive Elphick\thanks{\texttt{clive.elphick@gmail.com}}, Miriam Farber\thanks{\texttt{Department of Mathematics, Massachusetts Institute of Technology, Cambridge MA, USA, mfarber@mit.edu}. The work of this author was supported by the National Science Foundation Graduate Research Fellowship under Grant No. 1122374}, Felix Goldberg\thanks{\texttt{Caesarea-Rothschild Institute, University of Haifa, Haifa, Israel, felix.goldberg@gmail.com}} and Pawel Wocjan\thanks{\texttt{Department of Electrical Engineering and Computer Science, University of Central
Florida, Orlando, USA, wocjan@eecs.ucf.edu}}}
\maketitle

\abstract{A well known upper bound for the spectral radius of a graph, due to Hong, is that $\mu_1^2 \le 2m - n + 1$. It is conjectured that for connected graphs $n - 1 \le s^+ \le 2m - n + 1$, where $s^+$ denotes the sum of the squares of the positive eigenvalues. The conjecture is proved for various classes of graphs, including bipartite, regular, complete $q$-partite, hyper-energetic, and barbell graphs. Various searches have found no counter-examples. The paper concludes with a brief discussion of the apparent difficulties of proving the conjecture in general. }

\section{Introduction}

Let $G$ be a simple and undirected graph with $n$ vertices, $m$ edges, chromatic number $\chi$, minimum degree $\delta$, maximum degree $\Delta$ and adjacency matrix $A$ with eigenvalues $\mu_1 \ge \mu_2\ge ... \ge \mu_n$. The inertia of A is the ordered triple $(\pi, \nu, \gamma)$, where $\pi, \nu$ and $\gamma$ are the numbers (counting multiplicities) of positive, negative and zero eigenvalues of A respectively. Let

\[
s^+ = \sum_{i=1}^\pi \mu_i^2 \mbox{  and  } s^- = \sum_{i=n-\nu+1}^n \mu_i^2.
\]

Note that $\sum_{i=1}^n \mu_i^2 = s^+ + s^- = tr(A^2) = 2m$ and $2m \ge 2(n - 1)$ for connected graphs. Also let graph energy $E = \sum_{i=1}^n |\mu_i|.$ Since $tr(A)=0$,
$$
\sum_{i=1}^{\pi} \mu_i = -\sum_{i=n-\nu+1}^n \mu_i = E/2.
$$
Wocjan and Elphick \cite{wocjan13} proved that $\chi \ge s^+/s^-$ and conjectured that $\chi \ge 1 + s^+/s^-$. This Conjecture was recently proven by Ando and Lin in \cite{ando15}. It provides an example of replacing $\mu_1^2$ with $s^+$, because Edwards and Elphick \cite{edwards83} proved that $\chi \ge 2m/(2m - \mu_1^2)$.

In 1988 Hong \cite{hong88} proved that for connected graphs:

\[
\mu_1^2 \le 2m - n + 1,
\]

with equality only for $K_n$ and Star graphs. Note that for $K_n$ and Star graphs, $s^+ = \mu_1^2$. Hong \cite{hong93} also noted that this bound holds for graphs with no isolated vertices.

This bound has been strengthened by several authors. For example, Nikiforov \cite{nikiforov02} proved that:

\[
\mu_1 \le \frac{\delta - 1}{2} + \sqrt{2m - n\delta + \frac{(1 + \delta)^2}{4}}
\]

which is exact for regular graphs and strengthens Hong's bound, as discussed in \cite{nikiforov02}.

\section{Conjecture}

\begin{con}
Let $G$ be a connected graph. Then

\[
\min{(s^-, s^+)} \ge n - 1.
\]
\end{con}

Note that $s^- \ge n - 1$ implies $s^+ \le 2m - n + 1$ and vice versa.

\begin{con}
Let $G$ be a graph with $\kappa$ connected components. Then

\[
\min{(s^-, s^+)} \ge n - \kappa.
\]

\end{con}

\begin{proof}
Let $G_1, ... , G_\kappa$ denote the components of $G$ and let $n_i$ denote the number of vertices in $G_i$. Then

\[
s^-(G) = \sum {s^-(G_i)} \ge \sum (n_i - 1) = n - \kappa,
\]

and similarly for $s^+(G)$.

\end{proof}

\subsection{Comments}

A graph is connected if and only if its adjacency matrix is irreducible. In the language of matrix algebra, this conjecture can therefore be expressed as $min{(s^-, s^+)} \ge n - 1$ for binary, symmetric, irreducible matrices with zero trace.

Note that if $L$ is the Laplacian of $G$, then $n - \kappa =  rank(G) = rank(L) = $ number of positive eigenvalues of $L$.

We have searched the 10,000s of connected named graphs with $6$ to $40$ vertices in Wolfram Mathematica, and all connected graphs with up to 8 vertices, and found no counter-examples. A reviewer of this paper has also
kindly checked all connected graphs with 9 and 10 vertices, and connected graphs with maximum degree four on 11 and 12 vertices and found no counter-example.

Note that for connected graphs, if $s^+ > s^-$ then $s^+ > m \ge n - 1$ and if $s^- > s^+$ then $s^- > m \ge n - 1$. Most, but not all graphs, have $s^+ \ge s^-$. So for any connected graph one half of the conjecture is true.

If we consider the set of connected graphs on $n$ vertices, then it is notable that $s^- = n - 1$ for the graphs with the minimum number of edges (Trees) and the maximum number of edges ($K_n$).
\begin{thm}
Let $G$ be any graph. Then $s^-(G) \le n^2/4$.
\end{thm}

\begin{proof}

We use that $\mu_1 \ge 2m/n$ and assume that $s^- > n^2/4$, in which case:

\[
2m = s^+ + s^- \ge \mu_1^2 + s^- \ge \frac{4m^2}{n^2} + s^- > \frac{4m^2}{n^2} + \frac{n^2}{4}.
\]

This rearranges to:

\[
0 > \left(\frac{2m}{n} - \frac{n}{2}\right)^2
\]

which is a contradiction.
\end{proof}

Note that $s^- = \mu_n^2 = n^2/4$ for regular complete bipartite graphs. This bound can be compared with the following bound due to Constantine \cite{constantine85}:

\[
\mu_n^2 \le\left\lfloor \frac{n}{2} \right\rfloor \left\lceil \frac{n}{2} \right\rceil \le \frac{n^2}{4}.
\]

\subsection{An alternative formulation}

The cyclomatic number, $c(G)$, is the minimum number of edges that need to be removed from a graph to make it acyclic. It is well known that $c = m - n + \kappa$, where $\kappa$ is the number of components of a graph.  We can therefore reformulate Conjecture 2 as follows:

\[
m - c \le s^- \le m + c
\]

and similarly for $s^+$. When $c = 0$, $G$ is a forest for which $s^- = s^+ = m$ so the conjecture is true. It may therefore be possible to prove the conjecture using induction on $c$.

\section{Bounds using graph energy}

\begin{lem}

Let $\tau = |\mu_n|$ and $E$ denote the energy of a graph. Then

\[
s^- \le \frac{\tau E}{2}.
\]

\end{lem}

\begin{proof}
\[
 s^- = \sum_{i=n-\nu+1}^n \mu_i^2 \le \mu_n \sum_{i=n-\nu+1}^n \mu_i = \frac{\tau E}{2}.
\]

\end{proof}

Similarly, $s^+ \le \mu_1E/2$.

\begin{lem}\label{lem:lowerboundsm}
Let $\nu$ denote the number of negative eigenvalues of a graph with energy $E$. Then:

\[
s^- \ge \frac{E^2}{4\nu}.
\]

\begin{proof}

Using Cauchy-Schwartz:

\[
s^- = \sum_{i=n-\nu+1}^n \mu_i^2 \ge \frac{(\sum_{i=n-\nu+1}^n \mu_i)^2}{\nu} = \frac{E^2}{4\nu}.
\]

\end{proof}

\end{lem}

Similarly, $s^+ \ge E^2/4\pi$.

\begin{lem}\label{lem:boundnu}
Let $G$ be a graph for which $m \ge \nu(n - 1)$, where $\nu$ is the number of negative eigenvalues. Then

\[
s^- \ge n - 1.
\]

\begin{proof}
Brualdi \cite{brualdi06} proved that $E \ge 2\sqrt{m}$. Therefore using Lemma~\ref{lem:lowerboundsm}:

\[
s^-  \ge \frac{E^2}{4\nu} \ge \frac{m}{\nu} \ge n - 1.
\]

\end{proof}

\end{lem}

Similarly if $m \ge \pi(n - 1)$ then $s^+ \ge n - 1$.

\section{Proofs for various classes of graphs}

In this section we prove the conjecture for bipartite, regular, complete $q$-partite, hyperenergetic, and barbell  graphs. We have also proved that $s^{-} \geq n-1$ for graphs with precisely two negative eigenvalues and smallest degree at least 2, and we present it in section 5.

The proof for barbells is of interest, since $2K_k$ is an example of a disconnected graph for which $s^- = 2k - 2 < n - 1$. The "closest" connected graph to $2K_k$ is the barbell on $2k$ vertices. The proof for hyperenergetic graphs is of interest because almost all graphs are hyperenergetic.

\subsection{Bipartite graphs}

\begin{thm}\label{thm:bipart}
Let $G$ be a connected bipartite graph. Then $\min{(s^-, s^+)} \ge n - 1$.

\end{thm}
\begin{proof}
The spectrum of bipartite graphs is symmetrical about zero. Therefore:

\[
s^+ = s^- = m \ge n - 1.
\]

Note that for Trees, $m = n - 1$ so for these bipartite graphs the conjecture is exact.
\end{proof}
Note that in the theorem above, it is enough to assume that $G$ is a bipartite graph with at least $n-1$ edges.
\subsection{Regular graphs}

\begin{thm}
Let $G$ be a connected regular graph. Then $\min(s^-, s^+) \ge n - 1$.

\end{thm}

\begin{proof}

Ando and Lin \cite{ando15} proved a conjecture due to Wocjan and Elphick \cite{wocjan13} that:

\[
1 + \frac{s^+}{s^-} \le \chi(G) \mbox{ and that } 1 + \frac{s^-}{s^+} \le \chi(G).
\]

Brooks \cite{brooks41} proved that if $G$ is a connected graph and is neither an odd cycle nor a complete graph, then $\chi \le \Delta$, where $\Delta$ is the maximum degree.

Therefore if $G$ is  a $d-$regular connected graph and neither an odd cycle nor a complete graph then:

\[
s^- \ge \frac{s^- + s^+}{\chi(G)} = \frac{2m}{\chi(G)} \ge \frac{2m}{\Delta} = \frac{2m}{d} = n.
\]

If $G$ is a complete graph then $s^- = n - 1$.

If $G$ is an odd cycle then $2m = 2n$. If $G = C_5$ then $s^- = 5.236$ which falls between $n - 1$ and $n + 1$. For larger odd cycles $s^-$ and $s^+$ rapidly converge to $n$.

A very similar proof is used to demonstrate $s^+ \ge n - 1$.
\end{proof}

\subsection{Complete $q$-partite graphs}

\begin{thm}
Let $G$ be a complete $q$-partite graph. Then $\min{(s^-, s^+)} \ge n - 1$.
\label{thm:completeqpartite}
\end{thm}

\begin{proof}
A complete $q$-partite graph has precisely one positive eigenvalue. Therefore using Hong's \cite{hong88} bound

\[
s^+ = \mu_1^2 \le 2m - n + 1 \mbox{ which implies that } s^- \ge n - 1.
\]

Using Lemma~\ref{lem:boundnu}, $\pi = 1$ so $m \ge \pi(n - 1)$ and hence $s^+ \ge n -1$.

\end{proof}
\subsection{Hyper-energetic graphs}

A graph is said to be hyper-energetic if $E > 2(n - 1)$. Nikiforov \cite{nikiforov07} proved that almost all graphs are hyper-energetic and many classes of graphs are hyper-energetic. As a result, Kneser graphs and their complements, Paley and ciculant graphs, line graphs of regular graphs and line graphs of any graph with $m > 2n - 1$ all satisfy the conjecture.

\begin{thm}
Let $G$ be a hyper-energetic graph. Then $\min{(s^-, s^+)} > n - 1. $

\begin{proof}
Using Lemma~\ref{lem:lowerboundsm}

\[
\min{(s^-, s^+)} \ge \min{\left(\frac{E^2}{4\nu}, \frac{E^2}{4\pi}\right)} > (n - 1)^2 \min{\left(\frac{1}{\nu}, \frac{1}{\pi}\right)} \ge n - 1.
\]

\end{proof}

\end{thm}

In the theorem above, it is actually enough to assume that $E \geq 2(n - 1)$, as the same proof (with weak inequality instead of strict) works in this case as well. In fact, the following stronger version of the theorem holds as well:
\begin{thm}
Let $G$ be a graph for which $E\geq 2n-3$. Then $\min{(s^-, s^+)} \geq n - 1. $

\begin{proof}
We already proved that $\min{(s^-, s^+)} > n - 1$ for complete graphs, so assume that $G \neq K_n$. Then $\max(\pi,\nu) \leq n-2$. Using Lemma~\ref{lem:lowerboundsm}, we have

\[
\min{(s^-, s^+)} \ge \min{\left(\frac{E^2}{4\nu}, \frac{E^2}{4\pi}\right)} \geq (n - 1.5)^2 \min{\left(\frac{1}{\nu}, \frac{1}{\pi}\right)} \geq \frac{(n - 1.5)^2}{n-2} \ge n - 1.
\]

\end{proof}
\end{thm}
\subsection{Barbell graphs}

\begin{thm}
Let $G$ be the barbell graph on $n = 2k$ vertices, with $k \ge 3$. Then $min{(s^-, s^+)} \ge n - 1. $

\end{thm}

\begin{proof}
The characteristic polynomial of the barbell with $2k$ vertices is \cite{cvetkovic95}(Theorem 2.11)

\[
(x + 1)^{2k-4}[(x + 1)^2(x - k + 1)^2 - (x - k + 2)^2]
\]

which simplifies to

\[
(x + 1)^{2k-4}[x^4 + (4 - 2k)x^3 + (k^2 - 6k + 5)x^2 + (2k^2 - 4k)x + 2k - 3].
\]

Hence its eigenvalues are

\[
\frac{1}{2}[k - 1 - \sqrt{5 - 2k + k^2}] \mbox{  ,  } \frac{1}{2}[k - 1 + \sqrt{5 - 2k + k^2}]
\]
\[
\frac{1}{2}[k - 3 - \sqrt{-3 + 2k + k^2}] \mbox{  ,  } \frac{1}{2}[k - 3 + \sqrt{-3 + 2k + k^2}]
\]

and $-1$ with multiplicity $2k - 4$.

\subsubsection{$s^+ \ge n - 1$}

$G$  has precisely two positive eigenvalues. Therefore using Lemma~\ref{lem:boundnu}

\[
\pi(n - 1) = 2(n - 1) \le m \mbox{ for } k \ge 5
\]

so $s^+ \ge n - 1$ for $k \ge 5$. It is easy to verify that $s^+ \ge n - 1$ for $k = 3$ and $k = 4$.

\subsubsection{$s^- \ge n - 1$}

We are seeking to prove that

\[
s^- = 2k - 4 + \frac{(k - 1 - \sqrt{5 - 2k + k^2})^2}{4} + \frac{(k - 3 - \sqrt{-3 + 2k + k^2})^2}{4} \ge 2k - 1 = n - 1.
\]

This simplifies to :
\[
(k - 1 - \sqrt{5 - 2k + k^2})^2 + (k - 3 - \sqrt{-3 + 2k + k^2})^2 > (k - 3 - \sqrt{-3 + 2k + k^2})^2 \ge 12.
\]

If $k = 3$ then $(k - 3 - \sqrt{-3 + 2k + k^2})^2$ = $12$ and for $k \ge 4$ it is straightforward that $(k - 3 - \sqrt{-3 + 2k + k^2})^2 > 12.$
\end{proof}

\section{Graphs with two negative eigenvalues}
In this section, we deal with graphs with two negative eigenvalues. Our main result in this section is Theorem~\ref{thm:maintwoeigs}, which states that graphs with exactly two negative eigenvalues and smallest degree at least 2 satisfy $s^{-} \geq n-1$. As we will mention later, the case of exactly one positive or one negative eigenvalue satisfies the conjecture, and hence it is natural to consider the 2 negative eigenvalues case. We will also prove two additional lemmas that are also of independent interest: Lemma~\ref{lem:canonicalkgraph} deals with the inequality $s^{-} \geq n-1$ for an additional class of graphs, and in Lemma~\ref{lem:improvedbound} we obtain a stronger version of Lemma~\ref{lem:boundnu}.
Let us start with the latter lemma. For a real symmetric matrix $A$, we denote by $PO(A)$ ($NE(A)$) the sum of the positive (negative) eigenvalues of $A$.

The following result is from \cite{LiWang98}:
\begin{thm}\label{thm:sumofpostwo}
Let $A$ be a real symmetric matrix whose trace is $Tr(A)$. Then $$2PO(A)^2 \geq Tr(A^2)+ (2PO(A)-Tr(A))Tr(A).$$
\end{thm}
Using the proof of theorem~\ref{thm:sumofpostwo} as it appears in \cite{LiWang98}, it is possible to derive an exact expression for $2PO(A)^2$. Since we are interested only in adjacency matrices of simple graphs, we will assume that $Tr(A)=0$.
\begin{thm}\label{thm:sumofposthree}
Let $A$ be a real symmetric matrix of order $n$ whose trace is 0, and let $(p,q,n-p-q)$ be the inertia of $A$. Denote by $\{\lambda_i\}_{i=1}^{p+q}$ the nonzero eigenvalues of $A$, such that $\lambda_1 \geq \lambda_2 \geq \ldots \geq \lambda_p>0>\lambda_{p+1} \geq \ldots \geq \lambda_{p+q}$. In addition, let
\begin{center}
$B_A=\sum \limits_{1 \leq i < j \leq p}\lambda_i\lambda_j+\sum \limits_{1 \leq i < j \leq q}\lambda_{p+i}\lambda_{p+j}$.
\end{center}
Then $2NE(A)^2=2PO(A)^2 = Tr(A^2)+ 2B_A$.
\end{thm}

We can now prove the following lemma:
\begin{lem}\label{lem:improvedbound}
Let $G$ be a graph for which $m \geq \nu(n-1) -B_G$ ($m \geq \pi(n-1) -B_G$), where $\nu$ ($\pi$) is the number of negative (positive) eigenvalues and $B_G:=B_A$. Then $s^{-} \geq n-1$ ($s^{+} \geq n-1$).
\end{lem}
\begin{proof}
From Theorem~\ref{thm:sumofposthree} we have $2NE(A)^2= Tr(A^2)+ 2B_G=2m+2B_G$, where $A$ is the adjacency matrix of $G$. Therefore by Cauchy-Schwartz and the assumption on $m$, we have
$$\nu s^{-}\geq NE(A)^2=m+B_G \geq \nu(n-1) -B_G+B_G=\nu(n-1),$$
so $s^{-} \geq n-1$. The proof for $s^{+}$ is similar.
\end{proof}

We will now introduce some additional notations. Let $Q(t)$ ($P(t)$) be the set of all connected graphs with $\pi=t$ ($\nu=t$).  It was shown in \cite{Smith} that $G \in Q(1)$ if and only if $G$ is a complete multipartite graph, and as a consequence it is not hard to show that $G \in P(1)$ if and only if it is a complete bipartite graph.
Let us introduce the definition of canonical graphs \cite{TorgaSev852}. For a graph $G=(V,E)$, consider an equivalence relation $\sim$ on $V$ in which $x \sim y$ for $x,y \in V$ if and only if they are not adjacent and they have the same neighbours in $G$. The quotient graph $g$ of $G$ under this relation is called the canonical graph of $G$, and if $G=g$ then $G$ is called canonical graph. Given a canonical graph $G$ of order $n$ with vertex set $V(G)=\{1,2,3,\ldots,n\}$, we denote by $EQ(G)$ the set of graphs whose canonical form is $G$. By definition, the vertex set of any graph $H \in EQ(G)$ can be partitioned into $n$ disjoint sets, such that each set corresponds to a particular vertex of $G$. For a graph $H \in EQ(G)$, we denote the size of the set that corresponds to the vertex $i$ ($1 \leq i \leq n$) by $a_i$. For example, the graph in Figure~\ref{fig:classofG3} belongs to the equivalence class of the graph $G_3$ from Figure~\ref{fig:graphs}, and we have $a_1=2, a_2=1, a_3=1, a_4=3$.

\begin{figure}[h]
\center
\includegraphics[height=1in,width=2.3in]{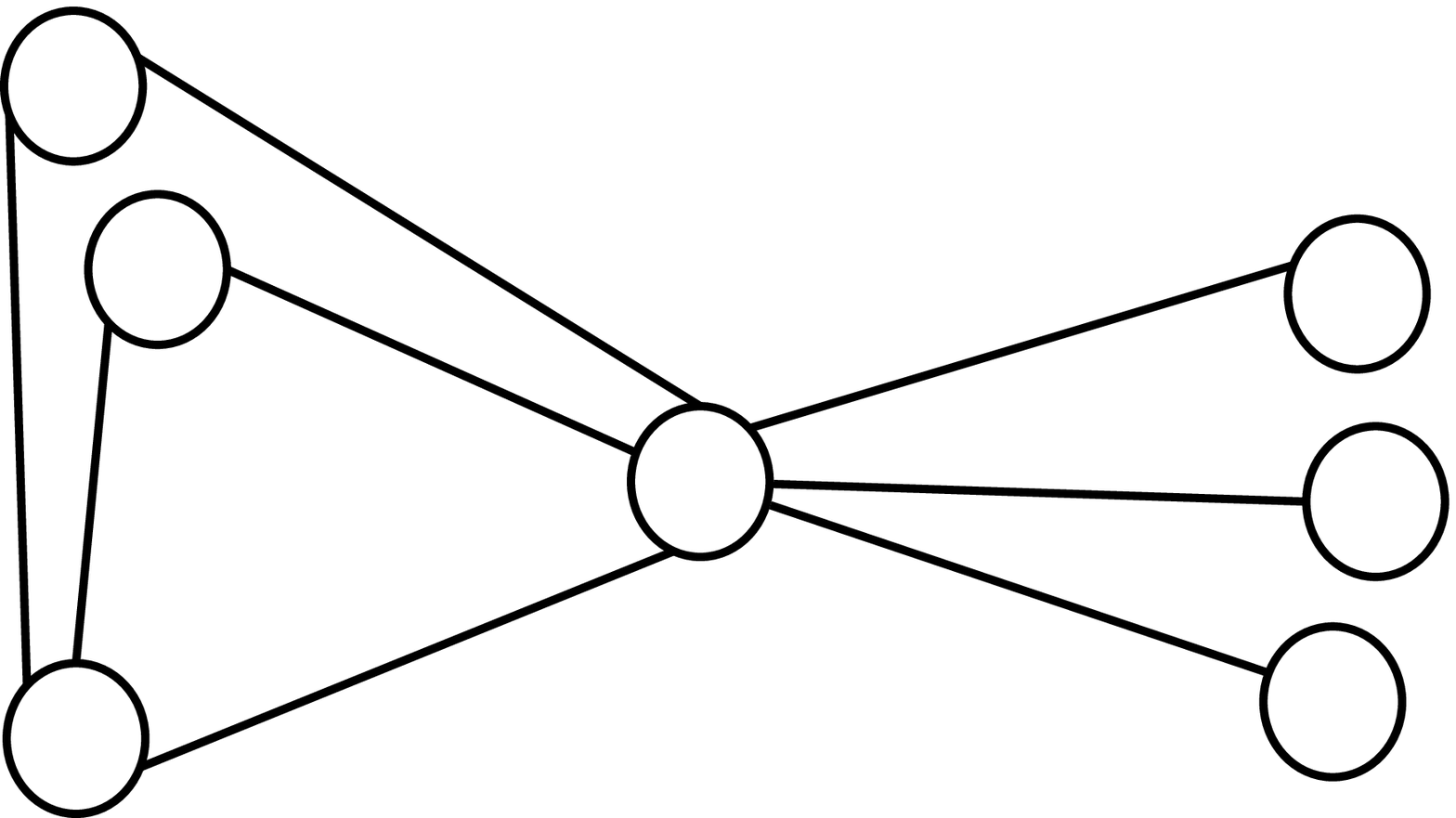}
\caption{example}
\label{fig:classofG3}
\end{figure}

\begin{lem}\cite{TorgaSev852}
For an arbitrary graph $G$ and its canonical graph $g$, the following equalities hold: $\pi(G)=\pi(g) $, $ \nu(G)=\nu(g)$.
\end{lem}
\begin{thm}\cite{TorgaSev852}
For any $t \geq 1$, there exists only finitely many canonical graphs in the class $P(t)$.
\label{finitecanonical}
\end{thm}
Using the characterisation of $P(1)$ and $Q(1)$ given above and Theorem~\ref{thm:completeqpartite} we get that if $G$ belongs to one of these classes, $min(s^{-},s^{+}) \geq n-1$. Hence, it is natural to consider the case $t=2$. The set of canonical graphs in $P(2)$ was fully characterised by Torgasev \cite{TorgaSev85}.

\begin{thm}\label{thm:twonegeigs}
A graph $G$ has exactly 2 negative eigenvalues if and only if its canonical graph is one of the graphs $G_i$, $i=1,\ldots, 9$ described below: $G_1=K_3$, $G_2=P_4$, $G_4=P_5$, $G_5=C_5$, and the remaining graphs are presented in Figure~\ref{fig:graphs}.
\end{thm}

\begin{figure}
\centering
\begin{subfigure}{.5\textwidth}
  \centering
  \includegraphics[height=.6in]{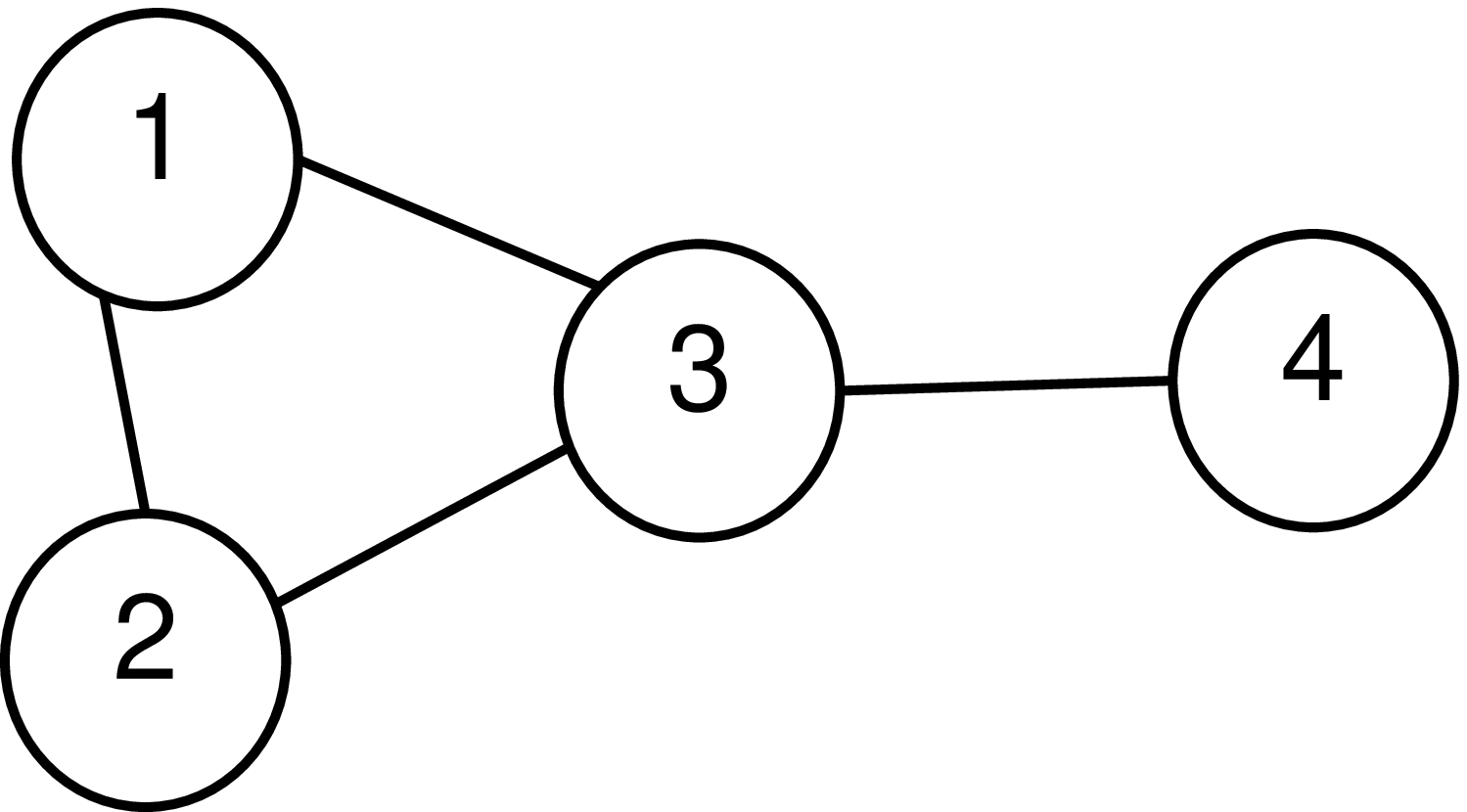}
  \caption{$G_3$}
\end{subfigure}%
\begin{subfigure}{.5\textwidth}
  \centering
  \includegraphics[height=.7in]{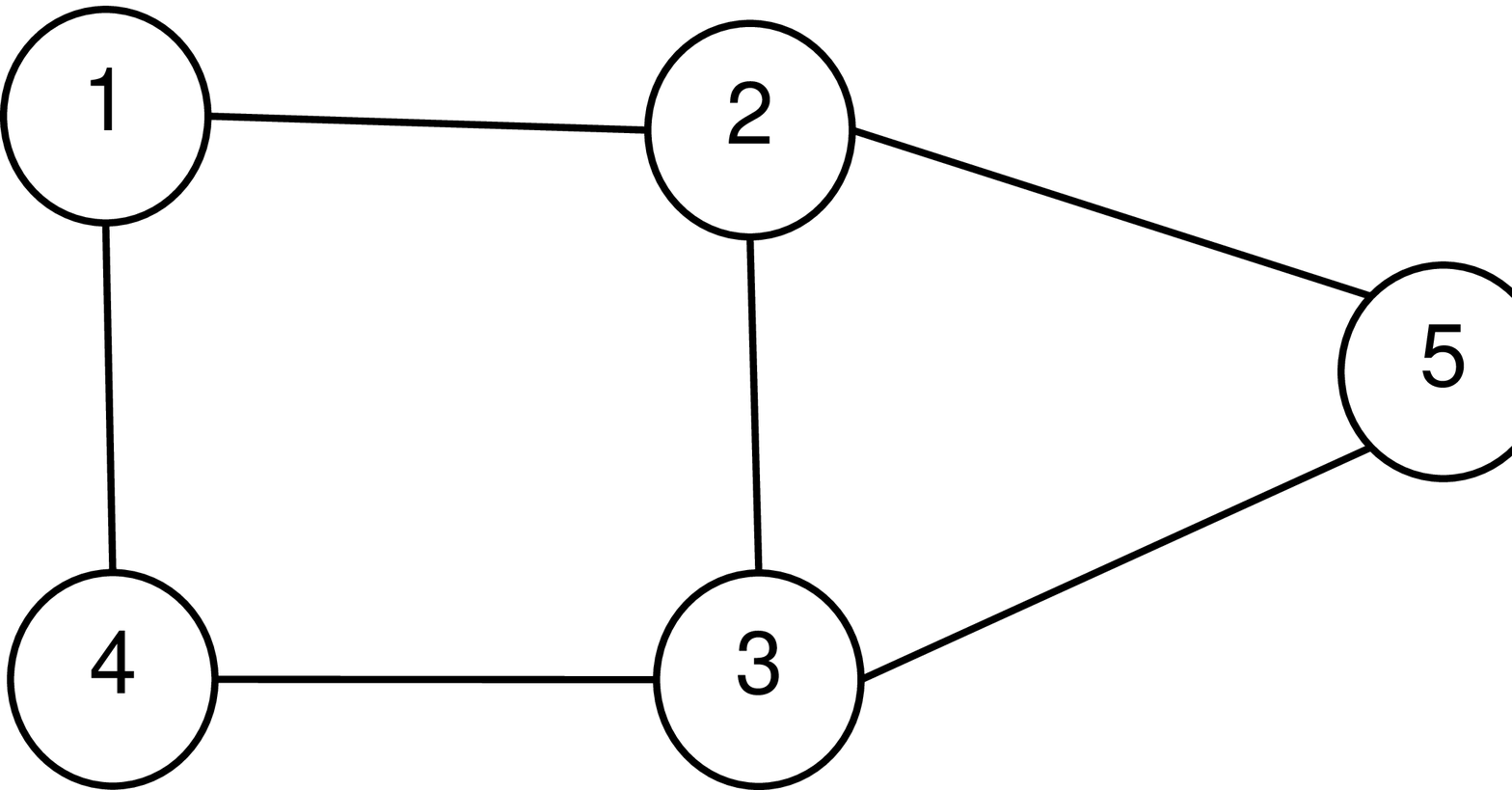}
  \caption{$G_6$}
\end{subfigure}
\begin{subfigure}{.5\textwidth}
  \centering
  \includegraphics[height=.7in]{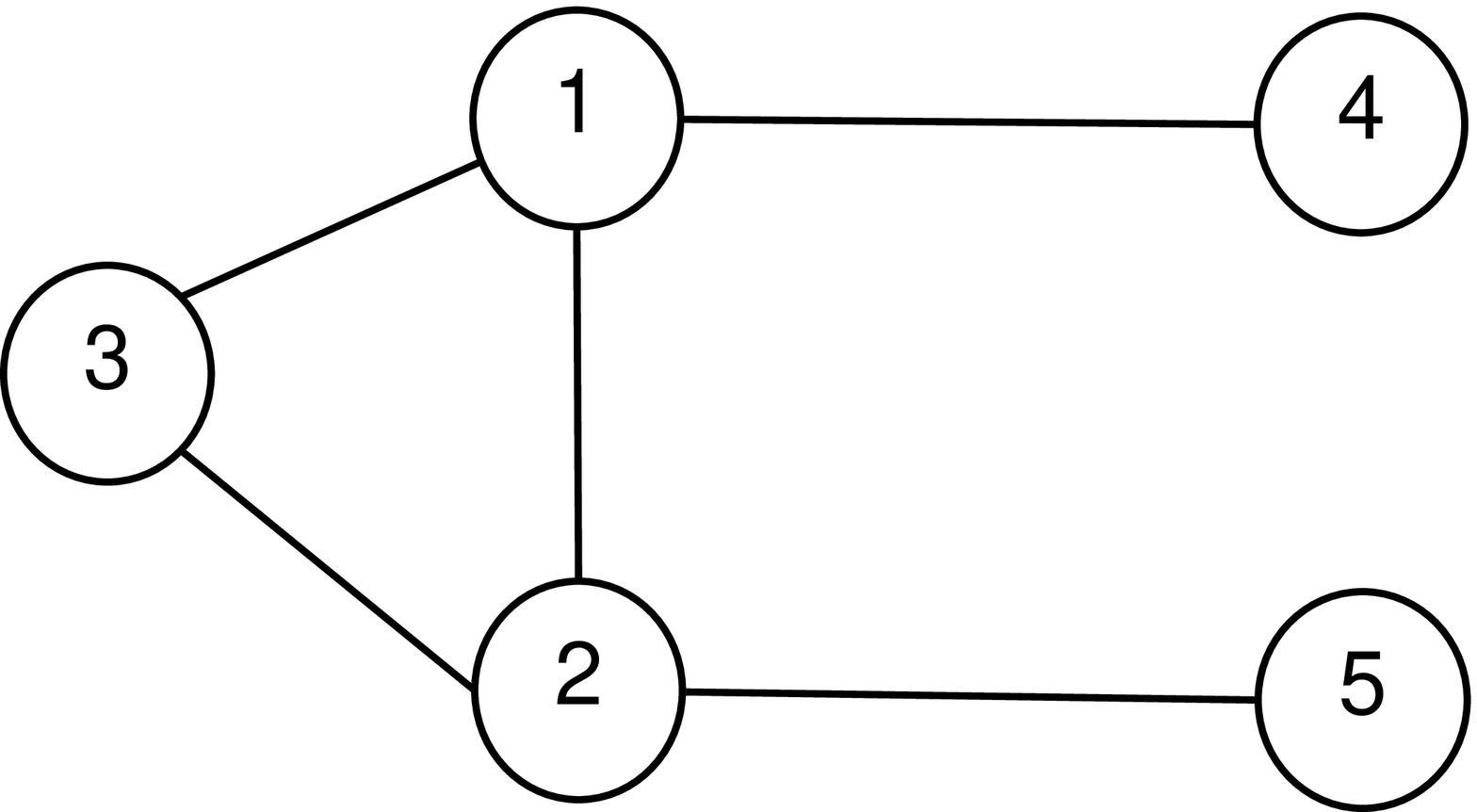}
  \caption{$G_7$}
\end{subfigure}
\begin{subfigure}{.5\textwidth}
  \centering
  \includegraphics[height=.8in]{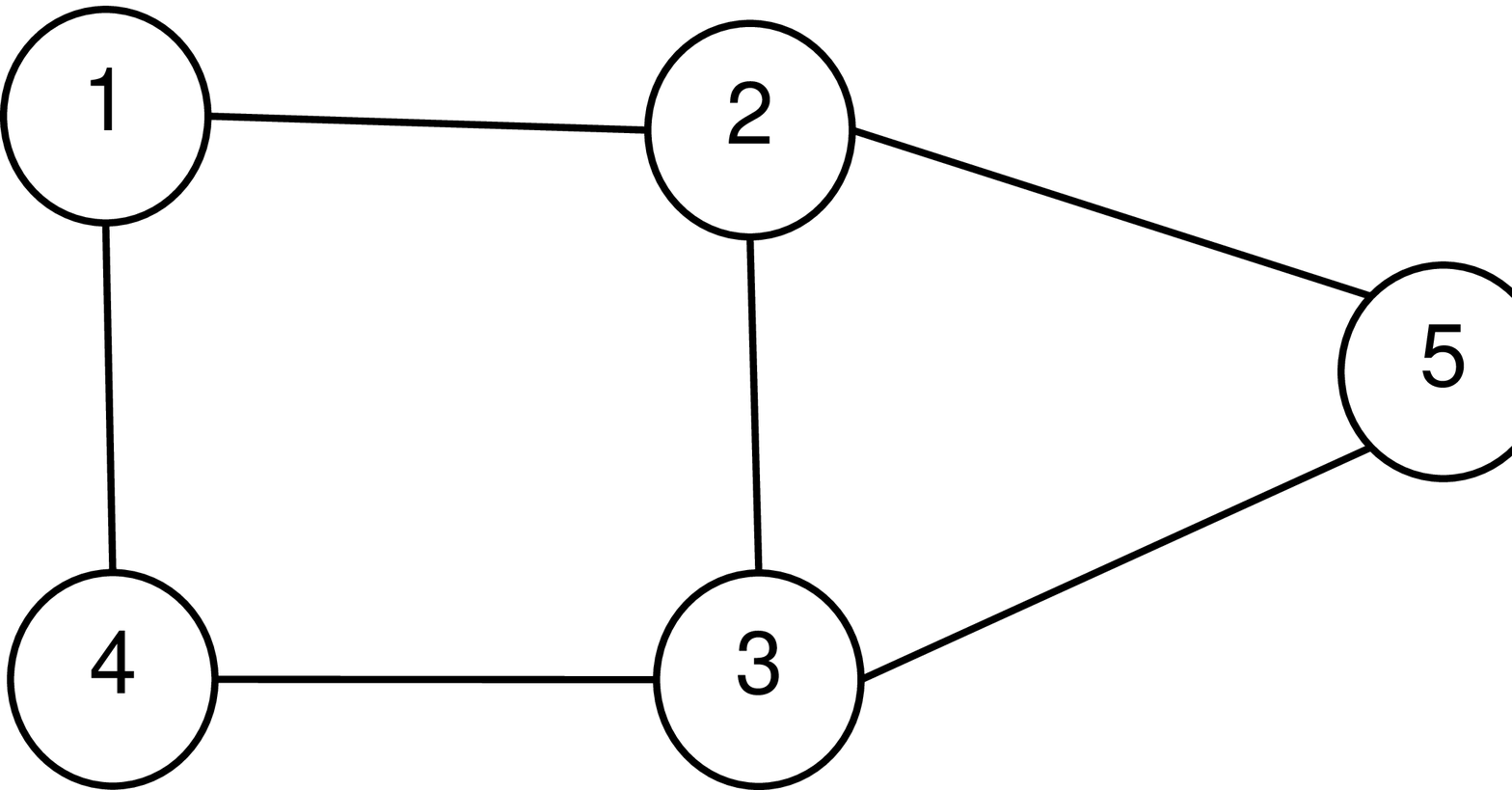}
  \caption{$G_8$}
\end{subfigure}
\begin{subfigure}{.5\textwidth}
  \centering
  \includegraphics[height=1in,width=2.7in]{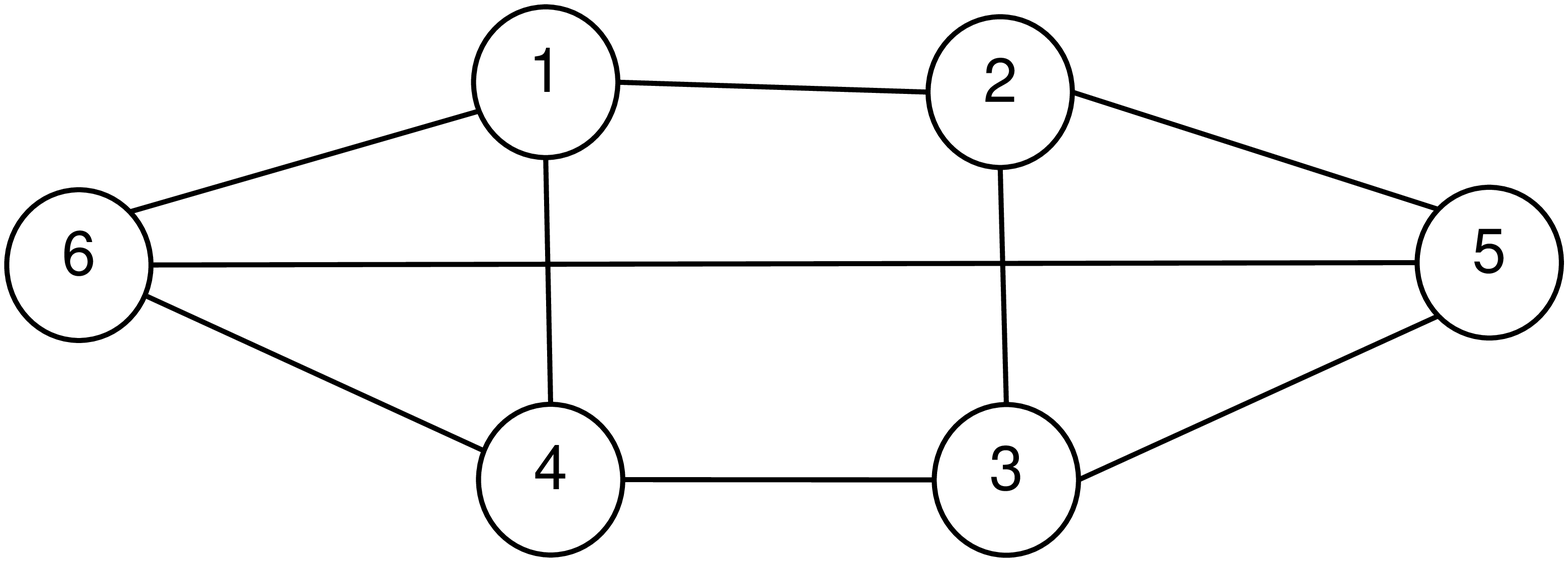}
  \caption{$G_9$}
\end{subfigure}
\caption{}
\label{fig:graphs}
\end{figure}

In this section we concentrate on this class, and our main result in this section is the following:

\begin{thm}\label{thm:maintwoeigs}
Let $G \in P(2)$ such that $\delta(G) \geq 2$. Then $s^{-} \geq n-1$
\end{thm}

The proof will be based on several lemmas, and we will actually show that $s^{-} \geq n-1$ also holds for many graphs in $P(2)$ with smallest degree equals 1, and also for many graphs in $P(t)$ for $t>0$. We start with a lemma that deals with certain type of graphs in $P(t)$- graphs whose canonical form has smallest degree at least $t$. This set is nonempty, as the complete graph on $t+1$ vertices is a canonical graph with smallest degree at least $t$. If $t=2$, then among the 9 canonical graphs in $P(2)$, 4 satisfy this property. We will show that for many of those graphs, $s^{-} \geq n-1$.

\begin{lem}\label{lem:canonicalkgraph}
Let $H \in P(t)$ be a canonical graph of order $n$ such that $\delta(H) \geq t$. Denote the vertex set of $H$ by $V(H)=\{1,2,\ldots,n\}$, and let $G \in EQ(H)$ such that $a_v \geq 2$ for all $v \in V(H)$. Then $G$ satisfies $s^{-} \geq n-1$.
\end{lem}
\begin{proof}
The graph $G$ has $n=\sum\limits_{v \in V(H)}a_v$ vertices and $m=\sum\limits_{\{u,v\}\in E(H)}a_ua_v$ edges. Since $a_v \geq 2$ for all $v \in V(H)$ and $\delta(H)\geq t$, we have
\begin{center}
$m=\sum\limits_{\{u,v\}\in E(H)}a_ua_v \geq \sum\limits_{\{u,v\}\in E(H)}(a_u+a_v) \geq \sum\limits_{v \in V(H)}ta_v=tn>t(n-1)$.
\end{center}
Thus $m \geq t(n-1)$, and using Lemma~\ref{lem:boundnu} we are done.
\end{proof}
\begin{lem}\label{lem:maintwoeigsfirst}
Let $G \in EQ(G_i)$ for some $i \in \{1,2,4\}$. Then $s^{-} \geq n-1$.
\end{lem}
\begin{proof}
Any graph in $EQ(G_1)$ is a complete 3-partite graph. Similarly, graphs that belong to $EQ(G_2)$ or $EQ(G_4)$ are bipartite. We already proved the claim for such classes of graphs in Theorems~\ref{thm:bipart} and~\ref{thm:completeqpartite}, so we are done.
\end{proof}
\begin{lem}\label{lem:maintwoeigssecond}
Let $G \in EQ(G_i)$ for some $i \in \{3,5,6,7,8,9\}$. Assume that for any vertex $v\in V(G_i)$, $a_v \geq 2$. Then $s^{-} \geq n-1$.
\end{lem}
\begin{proof}
For $i \in \{5,6,9\}$ the claim follows from Lemma~\ref{lem:canonicalkgraph}. We will consider the remaining 3 cases, and show that in each case we have $m \geq 2n$. Then we can apply Lemma~\ref{lem:boundnu}, and conclude that $s^{-} \geq n-1$.
\begin{enumerate}
  \item $G \in EQ(G_3)$. We have $m=a_1a_2+a_3a_1+a_2a_3+a_3a_4 $, and since $a_i \geq 2 $ for all $i$, we get $m=a_1a_2+a_3a_1+a_2a_3+a_3a_4 \geq 2a_2+2a_1+2a_3+2a_4=2n$.
  \item $G \in EQ(G_7)$. Then
\begin{center}
$m=a_1a_2+a_2a_3+a_3a_1+a_1a_4+a_2a_5 \geq 2a_2+2a_3+2a_1+2a_4+2a_5=2n$.
\end{center}
  \item $G \in EQ(G_8)$. Then
\begin{center}
$m=a_1a_2+a_2a_3+a_3a_4+a_4a_1+a_2a_5+a_3a_5+a_5a_6 \geq 2a_2+2a_3+2a_4+2a_1+2a_5+2a_6=2n$.
\end{center}
\end{enumerate}
Thus we showed that in all the cases $m \geq 2n$, so we are done.
\end{proof}

In order to examine the cases in which $a_i=1$ for some $i$, we will use Lemma~\ref{lem:improvedbound}. In general it seems quite hard to find the exact value of $B_G$, but for our purpose it will be enough to use lower bounds for $B_G$.
\begin{lem}\label{lem:maintwoeigsthird}
Let $G \in EQ(G_i)$ for some $i \in \{5,6,9\}$. Then $s^{-} \geq n-1$.
\end{lem}
\begin{proof}
First, note that if $a_v \geq 2$ for all $v \in V(G_i)$ then using Lemma~\ref{lem:maintwoeigssecond} we are done. Therefore, we can assume that $a_v=1$ for at least one of the vertices $v \in V(G_i)$. Also, note that for any pair of positive integers $c,d$ we have $cd \geq c+d-1$. We now divide the proof into 3 cases:
\begin{enumerate}
  \item $G \in EQ(G_5)$. Let us number the vertices of $G_5$ by 1,2,3,4,5 in clockwise order, and assume without loss of generality that $a_1=1$. We then have
      \begin{center}
      $m=a_2+a_2a_3+a_3a_4+a_4a_5+a_5$, $n-1=a_2+a_3+a_4+a_5$.
      \end{center}
      Using Lemma~\ref{lem:improvedbound}, it is enough to show that $m \geq 2(n-1)-B_G$.
      The eigenvalues of $G_5$ are $\{2,\frac{-1+\sqrt{5}}{2},\frac{-1+\sqrt{5}}{2},\frac{-1-\sqrt{5}}{2},\frac{-1-\sqrt{5}}{2}\}$, and hence $B_{G_5} \geq 5$. Therefore, from the interlacing Theorem, $B_G \geq 5$, and hence it is enough to show that $m \geq 2(n-1)-5$. Now,
$$
m \geq 2(n-1)-5
$$
if and only if
$$
a_2+a_2a_3+a_3a_4+a_4a_5+a_5 \geq 2a_2+2a_3+2a_4+2a_5 -5
$$
if and only if
$$
a_2a_3+a_3a_4+a_4a_5 \geq a_2+2a_3+2a_4+a_5 -5
$$
if and only if
$$
(a_2a_3-a_2-a_3+1)+(a_3a_4-a_3-a_4+1)+(a_4a_5-a_4-a_5+1) \geq -2
$$
which holds since each summand on the left hand side is at least 0.

  \item $G \in EQ(G_6)$. It is easy to check that $B_{G_6} \geq 2$. Now,
$$
m \geq 2(n-1)-2
$$
if and only if
$$
a_1a_2+a_2a_3+a_3a_4+a_4a_1+a_2a_5+a_3a_5 \geq 2a_1+2a_2+2a_3+2a_4+2a_5 -4
$$
if and only if
$$
(a_1a_2-a_1-a_2+1)+(a_2a_3-a_2-a_3+1)+(a_3a_4-a_3-a_4+1)+
$$
$$
+(a_1a_4-a_1-a_4+1)+a_5(a_2+a_3-2) \geq 0
$$
which holds, so we are done.
\item $G \in EQ(G_9)$. We have $B_{G_9}=7$, and hence in order to prove that $m \geq 2(n-1)-B_G$, it is enough to show that
\begin{center}
$a_1a_2+a_2a_5+a_5a_3+a_3a_4+a_4a_6+a_6a_1 \geq 2a_2+2a_5+2a_3+2a_4+2a_6+2a_1-9$.
\end{center}
This holds since
$$
(a_1a_2-a_1-a_2+1)+(a_2a_5-a_2-a_5+1)+(a_5a_3-a_5-a_3+1)+
$$
$$
+(a_3a_4-a_3-a_4+1)+(a_4a_6-a_4-a_6+1)+(a_6a_1-a_6-a_1+1) \geq 0,
$$
so we are done.
\end{enumerate}
\end{proof}
So far we have shown that for all $i \in \{1,2,4,5,6,9\}$ and $G \in EQ(G_i)$, we have $s^{-} \geq n-1$. In the following lemma we discuss the remaining three cases.
\begin{lem}\label{lem:maintwoeigsfourth}
Let $G \in EQ(G_i)$ for some $i \in \{3,7,8\}$. Assume that for all $v \in V(G_i)$ such that $v$ adjacent to a pendant vertex we have $a_v>1$. Then $s^{-} \geq n-1$.
\end{lem}
\begin{proof}
We divide the proof into three cases:
\begin{enumerate}
  \item $G \in EQ(G_3)$. In this case, vertex 3 is adjacent to a pendant vertex, so by our assumption $a_3 \geq 2$. Hence, in order to prove that $m \geq 2(n-1)-B_{G}$, it is enough to show $$a_1a_2+a_3a_1+a_2a_3+a_3a_4 \geq 2a_2+2a_1+2a_3+2a_4-4$$
      (keeping in mind that $B_{G_3}>2$). Since $a_3 \geq 2$, then $a_3a_4 \geq 2a_4$, and the inequality $a_1a_2+a_3a_1+a_2a_3 \geq 2a_2+2a_1+2a_3-4$ follows from
\begin{center}
$(a_1a_2-a_1-a_2+1)+(a_3a_1-a_3-a_1+1)+(a_2a_3-a_2-a_3+1) \geq 0$.
\end{center}
  \item $G \in EQ(G_7)$. We have $a_1 \geq 2, a_2 \geq 2$, $B_{G_7}>3$. Therefore, as before, it is enough to show that $a_1a_2+a_2a_3+a_3a_1 \geq 2a_2+2a_3+2a_1 -5$, which holds similarly to the previous case.
  \item $G \in EQ(G_8)$. We have $a_5 \geq 2$, $B_{G_8} \geq 2$. In this case, it is enough to show that
\begin{center}
$a_1a_2+a_2a_3+a_3a_4+a_4a_1+a_2a_5+a_3a_5 \geq 2a_2+2a_3+2a_4+2a_1+2a_5-4$,
\end{center}
 which holds since
\begin{center}
$(a_1a_2-a_1-a_2+1)+(a_2a_3-a_2-a_3+1)+(a_3a_4-a_3-a_4+1)+(a_4a_1-a_1-a_4+1)+(a_2+a_3-2)a_5 \geq 0$.
\end{center}
\end{enumerate}
\end{proof}

Finally, Theorem~\ref{thm:maintwoeigs} follows from Lemmas~\ref{lem:maintwoeigsfirst},~\ref{lem:maintwoeigssecond},~\ref{lem:maintwoeigsthird} and ~\ref{lem:maintwoeigsfourth}.

\section{Conclusions}

It seems unlikely that this conjecture can be proved using bounds on graph energy, because the conjecture is exact for trees but bounds on energy are not exact for trees. The difficulty in proving the conjecture appears to be that graph connectedness needs to be central to a proof. Graph connectedness is equivalent to irreducibility of the adjacency matrix of a graph. Much is known about the largest eigenvalue of an irreducible matrix using Perron-Frobenius theory, but much less is known about all eigenvalues of irreducible matrices. We have attempted a proof using the Lieb-Thirring inequalities \cite{lieb00}, which have applications in quantum mechanics, but without success.

\end{document}